\documentclass{amsart}
\usepackage{amsmath}
\usepackage{amssymb}
\usepackage{amsthm}
\usepackage{enumerate}
\usepackage[pdftex]{graphicx}
\usepackage{caption}
\theoremstyle{definition}
\newtheorem{definition}{Definition}[section]
\theoremstyle{plain}
\newtheorem{lemma}[definition]{Lemma}
\newtheorem{theorem}[definition]{Theorem}
\newtheorem{proposition}[definition]{Proposition}

\theoremstyle{remark}

\makeatletter
\@namedef{subjclassname@2020}{
  \textup{2020} Mathematics Subject Classification}
\makeatother

\newcommand{\myint}{\operatorname{int}}
\newcommand{\mycl}{\operatorname{cl}}
\newcommand{\mydcl}{\operatorname{discl}}
\newcommand{\mydclo}{\operatorname{dcl}}
\newcommand{\myacl}{\operatorname{acl}}
\newcommand{\myrk}{\operatorname{rk}}
\newcommand{\mydim}{\operatorname{D}}

\DeclareMathSymbol{\mlq}{\mathord}{operators}{``}
\DeclareMathSymbol{\mrq}{\mathord}{operators}{`'}

\begin{document}
\title[Pregeometry over locally o-minimal structure and dimension]{Pregeometry over locally o-minimal structure and dimension}
\author[M. Fujita]{Masato Fujita}
\address{Department of Liberal Arts,
Japan Coast Guard Academy,
5-1 Wakaba-cho, Kure, Hiroshima 737-8512, Japan}
\email{fujita.masato.p34@kyoto-u.jp}

\begin{abstract}
We define a discrete closure operation for definably complete locally o-minimal structures $\mathcal M$.
The pair of the underlying set of $\mathcal M$ and the discrete closure operation forms a pregeometry.
We define the rank of a definable set over a set of parameters using this fact.
A definable set $X$ is of dimension equal to the rank of $X$ over the set of parameters of a formula defining the set $X$.
The structure $\mathcal M$ is simultaneously a first-order topological structure.
The dimension rank of a set definable in the first-order topological structure $\mathcal M$ also coincides with its dimension.
\end{abstract}

\subjclass[2020]{Primary 03C64}

\keywords{locally o-minimal structure, pregeometry, dimension}

\maketitle

\section{Introduction}\label{sec:intro}
We first explain the notions of dimension which are applicable to o-minimal structures.
 See \cite{vdD} for o-minimal structures.
Firstly, the notion of a pregeometry is a central notion of geometric stability theory \cite{Pillay}.
The definable closure operation and the algebraic closure operation are often employed as the closure operations in pregeometries.
Pillay and Steinhorn demonstrated that the universe of an o-minimal structure together with the definable closure operation forms a pregeometry in \cite{PS}.
The notion of rank on a pregeometry is discussed in \cite{Pillay2, HP}.
It gives a definition of dimension of a set definable in an o-minimal structure.

Secondly, Pillay proposed the notion of first-order topological structures and defined an ordinal-valued dimension in \cite{Pillay3}.
We call it dimension rank in this paper.
An o-minimal structure is a first-order topological structure, and the dimension rank is a possible candidate of dimension of a definable set.
Thirdly, van den Dries proposed an alternative definition of model-theoretic dimension by giving axioms to be satisfied by a dimension function \cite{vdD2}.

Finally, when a structure admits a definable cell decomposition such as an o-minimal structure, another definition of dimension of a definable set is employed such as in \cite{vdD}.
A cell is definably homeomorphic to a definable open subset of $M^d$ for some nonnegative integer $d$, where $M$ is the universe of the o-minimal structure.
The dimension of the cell is defined as $d$.
The dimension of an arbitrary definable set is the maximum of the dimensions of cells contained in the given set.

Mathews demonstrated that the above four definitions of dimension coincide in \cite{Mathews}. 
We prove assertions similar to Mathews's when structures are definably complete locally o-minimal.

Local o-minimality is a variant of o-minimality \cite{TV}.
The author have investigated sets definable in definably complete locally o-minimal structures.
He first defined and investigated the dimension of a set definable in a locally o-minimal structure which admits local definable cell decomposition in \cite{Fuji, Fuji3}.
His definition of dimension in \cite{Fuji} does not assume that a definable set is locally decomposable into finitely many cells, but it coincides with the maximum dimension of cells contained in the definable set.
He gave another definition of dimension of a set definable in a definably complete locally o-minimal structure \cite{Fuji4}, and he and his collaborators clarified basic properties of dimension in \cite{Fuji4, FKK}. 
They showed that the dimension defined in \cite{Fuji4} coincides with the dimension defined in \cite{Fuji}, and it satisfies the requirements of dimension function given in \cite{vdD2}.
In this paper, we demonstrate the equivalence among our dimension, the dimension rank on a first-order topological structure and the rank on a pregeometry in the definbly complete locally o-minimal setting.

The pair of the universe of an o-minimal structure and the definable closure operation forms a pregeometry.
In this case, a definable set is of dimension zero if and only if it is a finite set.
There exists an infinite set of dimension zero when the structure is defnably complete locally o-minimal.
We cannot adopt the definable closure as the closure operation in our case.
Therefore, we develop the notion of `discrete closure' in this paper.
We show that the universe of a defiably complete locally o-minimal structure together with the discrete closure operation forms a pregeometry.

This paper is organized as follows:
We recall the properties of definably complete locally o-minimal structures in Section \ref{sec:preliminary}.
We also review the definition of pregeometries in this section.
Section \ref{sec:pregeometry} is the main part of this paper.
We define the discrete closure and show that it enjoys the exchange property.
It implies that the pair of the universe and the discrete closure operation forms a pregeometry.
Theorem \ref{thm:rank} says that our definition of dimension of a definable set coincides with the rank on the pregeometry defined above.
Finally, we treat the dimension rank on a first-order topological structure in Section \ref{sec:first-order}.

In the last of this section, we summarize the notations and terms used in this paper.
The interior, closure and frontier of a subset $A$ of a topological space is denoted by $\myint(A)$, $\overline{A}$ and $\partial A$, respectively.
For any set $S$, let $|S|$ denote the cardinality of $S$.

The term `definable' means `definable in the given structure with parameters' in this paper unless we clearly specify the parameters.
Let $\mathcal L$ be a language and $\mathcal M=(M,\ldots)$ be an $\mathcal L$-structure.
Let $\Phi(\overline{x})$ be a first-order $\mathcal L(M)$-formula with $n$ free variables and $S \subseteq M^n$ be the set defined by the formula $\Phi(\overline{x})$.
The notation $\Phi(\mathcal M)$ denotes the definable set $\{\overline{x} \in M^n\;|\; \mathcal M \models \Phi(\overline{x})\}$; that is, $S=\Phi(\mathcal M)$.
When $\mathcal N=(N,\ldots)$ is an elementary extension of $\mathcal M$, the notations $\Phi(\mathcal N)$ and $S^{\mathcal N}$ denotes the definable set $\{\overline{x} \in N^n\;|\; \mathcal N \models \Phi(\overline{x})\}$.
The set $S^{\mathcal N}$ is uniquely determined independently of the choice of formulas $\Phi$ defining $S$. 

Consider a linearly ordered set without endpoints $(M,<)$.
An open interval is a nonempty set of the form $\{x \in M\;|\; a < x < b\}$ for some $a,b \in M \cup \{\pm \infty\}$.
It is denoted by $(a,b)$ in this paper.
When an expansion of a dense linear order without endpoints $\mathcal M=(M,<,\ldots)$ is given, the set $M$ equips the order topology induced from the order $<$. 
The Cartesian product $M^n$ equips the product topology of the order topology for each positive integer $n$.
An open box is the Cartesian product of open intervals.
We consider that $M^0$ is a singleton with the trivial topology, and a projection from $M^n$ onto $M^0$ is the trivial map.
A nonempty open box in $M^0$ is $M^0$ itself.

\section{Preliminary}\label{sec:preliminary}
In this section, we recall the results in the previous studies.

\subsection{On local o-minimality}
We first recall the definitions of local o-minimality and definably completeness.
\begin{definition}[\cite{TV}]
An expansion of a dense linear order without endpoints $\mathcal M=(M,<,\ldots)$ is \textit{locally o-minimal} if, for every definable subset $X$ of $M$ and for every point $a\in M$, there exists an open interval $I$ containing the point $a$ such that $X \cap I$ is  a finite union of points and open intervals.
\end{definition}

\begin{definition}[\cite{M}]
An expansion of a dense linear order without endpoints $\mathcal M=(M,<,\ldots)$ is \textit{definably complete} if any definable subset $X$ of $M$ has the supremum and  infimum in $M \cup \{\pm \infty\}$.
\end{definition}

It is already known that definably completeness and local o-minimality are preserved under elementary equivalence by \cite[Corollary 2.5]{TV} and \cite[p.1788, Proposition]{M}.
We call that a complete theory is a \textit{definably complete locally o-minimal theory} if every model is definably complete and locally o-minimal.
A complete theory having a definably complete locally o-minimal model is a definably complete locally o-minimal theory.

We review basic properties of definably complete locally o-minimal structures.
\begin{lemma}\label{lem:local}
Consider a definably complete locally o-minimal structure $\mathcal M=(M,<,\ldots)$.
Any definable subset of $M$ having an empty interior is closed and discrete. 
\end{lemma}
\begin{proof}
See \cite[Lemma 2.3]{Fuji4}.
\end{proof}

\begin{lemma}\label{lem:property_d}
Consider a definably complete locally o-minimal structure $\mathcal M=(M,<,\ldots)$.
Let $X$ be a definable subset of $M^n$ and $\pi: M^n \rightarrow M^d$ be a coordinate projection such that the the fibers $X \cap \pi^{-1}(x)$ are discrete for all $x \in \pi(X)$.
Then, there exists a definable map $\tau:\pi(X) \rightarrow X$ such that $\pi(\tau(x))=x$ for all $x \in \pi(X)$.
\end{lemma}
\begin{proof}
See \cite[Proposition 2.6]{FKK}.
\end{proof}

\begin{definition}[Local monotonicity]
A function $f$ defined on an open interval $I$ is \textit{locally constant} if, for any $x \in I$, there exists an open interval $J$ such that $x \in J \subseteq I$ and the restriction $f|_J$ of $f$ to $J$ is constant.
A function $f$ defined on an open interval $I$ is \textit{locally strictly increasing} if, for any $x \in I$, there exists an open interval $J$ such that $x \in J \subseteq I$ and $f$ is strictly increasing on the interval $J$.
We define a \textit{locally strictly decreasing} function similarly. 
\end{definition}

\begin{proposition}[Strong local monotonicity]\label{prop:monotonicity}
Consider a definably complete locally o-minimal structure $\mathcal{M}=(M,<,\ldots)$.
Let  $I$ be an interval and $f:I \to M$ be a definable function.
Then, there exists a mutually disjoint definable partition $I=X_d \cup X_c \cup X_+ \cup X_-$ satisfying the following conditions:
\begin{enumerate}
\item[(1)] the definable set $X_d$ is discrete and closed;
\item[(2)] the definable set $X_c$ is open and $f$ is locally constant on $X_c$;
\item[(3)] the definable set $X_+$ is open and $f$ is locally strictly increasing and continuous on $X_+$;
\item[(4)] the definable set $X_-$ is open and $f$ is locally strictly decreasing and continuous on $X_-$.
\end{enumerate}
\end{proposition}
\begin{proof}
See \cite[Theorem 2.3]{FKK}.
\end{proof}

In \cite{Fuji4}, the author demonstrated that 
\begin{itemize}
\item the family of sets definable in a definably complete o-minimal structure enjoys good dimension theory and 
\item any definable set is decomposed into good-shaped definable sets called quasi-special submanifolds  
\end{itemize}
under the assumption that the image of a nonempty definable discrete set under a coordinate projection is again discrete.
The author and his collaborators demonstrated that the assumption holds true for any definably complete locally o-minimal structures \cite[Theorem 2.5]{FKK}.
The above two results in \cite{Fuji4} are now available in definably complete locally o-minimal structures without the assumption.
We use this fact without notice in this paper.
We summarize the results of \cite{Fuji4} and \cite{FKK} below.

\begin{definition}[Dimension, \cite{Fuji4}]\label{def:dim}
Consider an expansion of a dense linear order without endpoints $\mathcal M=(M,<,\ldots)$.
Let $X$ be a nonempty definable subset of $M^n$.
The dimension of $X$ is the maximal nonnegative integer $d$ such that $\pi(X)$ has a nonempty interior for some coordinate projection $\pi:M^n \rightarrow M^d$.
We set $\dim(X)=-\infty$ when $X$ is an empty set.
\end{definition}

\begin{proposition}\label{prop:prelim_dim}
Consider a definably complete locally o-minimal structure $\mathcal M=(M,<,\ldots)$.
The following assertions hold true.
\begin{enumerate}
\item[(1)] A definable set is of dimension zero if and only if it is discrete.
When it is of dimension zero, it is also closed.
\item[(2)] Let $X$ and $Y$ be definable subsets of $M^n$.
We have 
\begin{align*}
\dim(X \cup Y)=\max\{\dim(X),\dim(Y)\}\text{.}
\end{align*}
\item[(3)] Let $f:X \rightarrow M^n$ be a definable map. 
We have $\dim(f(X)) \leq \dim X$.
\item[(4)] Let $X$ be a definable set.
We have $\dim (\partial X) < \dim X$ and $\dim (\overline{X})=\dim X$.
\item[(5)] Let $\varphi:X \rightarrow Y$ be a definable surjective map whose fibers are equi-dimensional; that is, the dimensions of the fibers $\varphi^{-1}(y)$ are constant.
We have $\dim X = \dim Y + \dim \varphi^{-1}(y)$ for all $y \in Y$.  
\end{enumerate}
\end{proposition}
\begin{proof}
See \cite[Proposition 2.8]{FKK}.
\end{proof}

\begin{definition}[\cite{Fuji4}]\label{def:quasi-special}
Consider an expansion of a densely linearly order without endpoints $\mathcal M=(M,<,\ldots)$.
Let $X$ be a definable subset of $M^n$ and $\pi:M^n \rightarrow M^d$ be a coordinate projection.
%
A definable subset is a \textit{$\pi$-quasi-special submanifold} or simply a \textit{quasi-special submanifold} if, $\pi(X)$ is a definable open set and, for every point $x \in \pi(X)$, there exists an open box $U$ in $M^d$ containing the point $x$ satisfying the following condition:
For any $y \in X \cap \pi^{-1}(x)$, there exist an open box $V$ in $M^n$ and a definable continuous map $\tau:U \rightarrow M^n$ such that $\pi(V)=U$, $\tau(U)=X \cap V$ and the composition $\pi \circ \tau$ is the identity map on $U$.

\end{definition}

\begin{proposition}\label{prop:quasi}
Let $\mathcal M=(M,<,\ldots)$ be a definably complete locally o-minimal structure and $A$ be a subset of $M$.
Let $X$ be a subset of $M^n$ definable over $A$.
There exists a family $\{C_i\}_{i=1}^N$ of mutually disjoint quasi-special submanifolds which are definable over $A$ such that $X= \bigcup_{i=1}^N C_i$.
\end{proposition}
\begin{proof}
It immediately follows from the proof of \cite[Lemma 4.3]{Fuji4}.
\end{proof}

\subsection{On pregeometry}
We recall the notion of a pregeometry.
It is found in \cite[Chapter 2]{Pillay} and \cite[Appendix C]{TZ}.

\begin{definition}[Pregeometry]\label{def:pregeometry}
Consider a set $S$ and a map $\mycl:\mathcal P(S) \rightarrow \mathcal P(S)$, where $\mathcal P(S)$ denotes the power set of $S$.
The pair $(S,\mycl)$ is a \textit{(combinatorial) pregeometry} if the following conditions are satisfied for any subset $A$ of $S$:
\begin{enumerate}
\item[(i)] $A \subseteq \mycl(A)$;
\item[(ii)] $\mycl(\mycl(A))=\mycl(A)$;
\item[(iii)] For any $a,b \in S$, if $a \in \mycl(A \cup \{b\}) \setminus \mycl(A)$, then $b \in \mycl(A \cup \{a\})$;
\item[(iv)] For any $a \in \mycl(A)$, there exists a finite subset $Y$ of $A$ such that $a \in \mycl(Y)$. 
\end{enumerate} 
The condition (iii) is called the \textit{exchange property}.

Consider a pregeometry $(S,\mycl)$.
Let $A$, $B$ and $C$ be subsets of $S$.
The set $A$ is $\mycl$-\textit{independent} over $B$ if, for any $a \in A$, we have $a \not\in \mycl((A \setminus \{a\}) \cup B)$.
A subset $A_0$ of $A$ is a $\mycl$-\textit{basis} for $A$ over $B$ if $A$ is contained in $\mycl(A_0 \cup B)$ and $A_0$ is $\mycl$-independent over $B$.
Each basis for $A$ over $B$ has the same cardinality and it is denoted by $\myrk^{\mycl}(A/B)$.
We simply denote it by $\myrk(A/B)$ when the closure operation $\mycl$ is clear from the context.
\end{definition}

The following result is well-known:

\begin{lemma}\label{lem:pregeometry}
Consider a pregeometry $(S,\mycl)$.
The symbols $A$, $B$, $C$ and $D$ denote arbitrary subsets of $S$.
The following assertions hold true.
\begin{enumerate}
\item[(i)] Two $\mycl$-bases for $A$ over $B$ have the same cardinality.
\item[(ii)] We have $\myrk(A/B) \leq \myrk(A/C)$ when $C$ is a subset of $B$.
\item[(iii)] When $B \subseteq C$ and $\myrk(A/B)=\myrk(A/C)$,
a $\mycl$-basis for $A$ over $B$ is a $\mycl$-basis for $A$ over $C$.
\end{enumerate}
\end{lemma}

\begin{definition}
Let $\mathcal L$ be a language and $\mathcal M=(M,\ldots)$ be an $\mathcal L$-structure.
Let $\mycl:\mathcal P(M) \rightarrow \mathcal P(M)$ be a closure operation such that the pair $(M,\mycl)$ is a pregeometry.
Consider subsets $A$ and $S$ of $M$ and $M^n$, respectively.
We define 
$$\myrk_{\mathcal M}^{\mycl}(S/A)=\max \{\myrk^{\mycl}(\{a_1,\ldots,a_n\}/A)\;|\; (a_1,\ldots, a_n) \in S\}.$$

Let $\mathcal L$ be a language and $T$ be its theory.
Consider a model $\mathcal M=(M,\ldots)$ of $T$ and a monster model $\mathbb M$ of $T$.
The universe of $\mathbb M$ is denoted by the same symbol $\mathbb M$ for simplicity.
Assume that there exists a closure operation $\mycl:\mathcal P(\mathbb M) \rightarrow \mathcal P(\mathbb M)$ such that the pair $(\mathbb M,\mycl)$ is a pregeometry. 
Let $A$ be a subset of $M$ and $S$ be a definable set.
We set $$\myrk_T^{\mycl} (S/A)=\myrk_{\mathbb M}^{\mycl}(S^{\mathbb M}/A).$$
We simply denote it by $\myrk(S/A)$ when $T$ and $\mycl$ are clear from the context.
%
\end{definition}

\section{Discrete closure and Pregeometry}\label{sec:pregeometry}
We are now ready to give the definition of discrete closure.

\begin{definition}[Discrete closure]\label{def:dcl}
Let $\mathcal L$ be a language containing a binary predicate $<$.
Consider a definably complete locally o-minimal $\mathcal L$-structure $\mathcal M=(M,<,\ldots)$.
The \textit{discrete closure} $\mydcl_{\mathcal M}(A)$ of a subset $A$ of $M$ is the set of points $x$ in $M$ having an $\mathcal L(A)$-formula $\phi(t)$ such that 
\begin{enumerate}
\item[(a)] the set $\phi(\mathcal M)$ contains the point $x$ and 
\item[(b)] it is discrete and closed.
\end{enumerate}
We often omit the subscript $\mathcal M$ of $\mydcl_{\mathcal M}(A)$.

Recall that a discrete, closed and definable subset of $M$ is a finite set when $\mathcal M$ is an o-minimal structure.
In this case, the discrete closure $\mydcl(A)$ of a subset $A$ of $M$ coincides with the algebraic closure $\myacl(A)$ and the definable closure $\mydclo(A)$.
\end{definition}

\begin{theorem}
Let $\mathcal L$ be a language containing the binary predicate $<$.
Consider a definably complete locally o-minimal $\mathcal L$-structure $\mathcal M=(M,<,\ldots)$.
The pair $(M,\mydcl)$ is a pregeometry.
\end{theorem}
\begin{proof}
%
The conditions (i) and (iv) in Definition \ref{def:pregeometry} are obviously satisfied.

We show that that the condition (ii) in Definition \ref{def:pregeometry} is satisfied.
The inclusion $\mydcl(A) \subseteq \mydcl(\mydcl(A))$ is obvious by Definition \ref{def:pregeometry}(i).
We demonstrate the opposite inclusion.
Take $b \in \mydcl(\mydcl(A))$.
There exists a tuple $a=(a_1,\ldots, a_m)$ of elements in $\mydcl(A)$ and an $\mathcal L(A)$-formula $\phi(x,y_1,\ldots, y_m)$ such that $\mathcal M \models \phi(b,a_1,\ldots, a_m)$ and the definable set $\{x \in M\;|\; \mathcal M \models \phi(x,a_1,\ldots, a_m)\}$ is discrete and closed.
Since $a_i \in \mydcl(A)$, there exists an $\mathcal L(A)$-formula $\psi_i(x)$ such that $\mathcal M \models \psi_i(a_i)$ and the definable set $\{x \in M\;|\; \mathcal M \models \psi_i(x)\}$ is discrete and closed for any $1 \leq i \leq m$.
Consider the $\mathcal L(A)$-formula $\eta(y_1,\ldots, y_m)$ representing that the set $$E_{y_1,\ldots, y_m}:=\{x \in M\;|\; \mathcal M \models \phi(x,y_1,\ldots, y_m)\}$$
is neither empty nor contains an open interval.
By Lemma \ref{lem:local}, we have $\mathcal M \models \eta(y_1,\ldots, y_m)$ if and only if the definable set $E_{y_1,\ldots, y_m}$ is nonempty, discrete and closed.
Consider the definable sets 
\begin{align*}
&F=\{(y_1,\ldots, y_m) \in M^m\;|\; \mathcal M \models \bigwedge_{i=1}^m \psi_i(y_i) \wedge \eta(y_1,\ldots, y_m)\} \text{ and }\\
&E =\{(x,y_1,\ldots,y_m) \in M^{m+1}\;|\; \mathcal M \models \phi(x,y_1,\ldots, y_m)\ \text{ and } (y_1,\ldots, y_m) \in F\}.
\end{align*}
By the assumption, the set $F$ is discrete and closed.
The set $F$ is the image of $E$ under the coordinate projection forgetting the first coordinate and its fibers are discrete and closed by the definition of the formula $\eta$.
 By Proposition \ref{prop:prelim_dim}(1),(5), $E$ is also discrete and closed.
The projection image $G$ of $E$ defined by
$$G=\{x \in M\;|\; \mathcal M \models \exists y_1, \ldots, y_m\ (x,y_1,\ldots, y_m) \in E\}$$
is also discrete and closed by Proposition \ref{prop:prelim_dim}(1),(3).
The definable set $G$ is defined by an $\mathcal L(A)$-formula by its construction, and we have $b \in G$.
It implies that $b \in \mydcl(A)$.
We have demonstrated that the condition (ii) is satisfied.

The remaining task is to demonstrate the exchange property (iii).
Fix elements $a,b \in M$ and a subset $A \subseteq M$ with $a \in \mydcl(A \cup \{b\}) \setminus \mydcl(A)$.
We demonstrate that $b \in \mydcl(A \cup \{a\})$.
Take an $\mathcal L(A)$-formula $\varphi(x,y)$ such that $D_b$ is discrete and closed, and it contains the point $a$.
Here, the notation $D_y$ denotes the set $\{x \in M\;|\; \mathcal M \models \varphi(x,y)\}$ for any $y \in M$.
Consider the definable sets
\begin{align*}
& S= \{y \in M\;|\; \mathcal M \models \exists x\  \varphi(x,y)\}\text{ and }\\
& T=\{y \in S\;|\; \text{ the set } D_y \text{ is discrete and closed}\}.
\end{align*}
We get $b \in T$ by the assumption.

We first consider the case in which there exist no open intervals contained in $T$ and containing the point $b$.
By local o-minimality, the point $b$ is either an isolated point of $T$ or an endpoint of a maximal open interval contained in $T$.
Consider the set of isolated points of $T$ and endpoints of maximal open intervals contained in $T$.
It is obviously defined by an $\mathcal L(A)$-formula and contains the point $b$.
It does not contain an open interval by its definition.
Therefore, it is discrete and closed by Lemma \ref{lem:local}.
It implies that $b \in \mydcl(A) \subseteq \mydcl(A \cup \{a\})$.

We next treat the case in which there exists an open interval $I$ contained in $T$ and containing the point $b$.
We may assume that $S$ is open, $S=T$ and the sets $D_y$ are discrete and closed for all $y \in S$ by considering the $\mathcal L(A)$-formula $\varphi(x,y) \wedge (y \in \myint(T))$ instead of $\varphi(x,y)$.
Let $\pi_i:M^2 \rightarrow M$ be the coordinate projections onto the $i$-th coordinates for $i=1,2$.
Set 
\begin{align*}
&\Gamma = \{(y,x) \in M^2\;|\; \mathcal M \models \varphi(x,y)\},\\
& U=\{(y,x) \in \Gamma\;|\; \text{ there exists an open box }B \subseteq M^2 \text{ such that } (y,x) \in B \text{ and }\\
&\qquad B \cap \Gamma \text{ is the graph of a monotone definable continuous function}\\
&\qquad\text{defined on }\pi_1(B)\} \text{ and }\\
&V = \Gamma \setminus U.
\end{align*}
We prove that $\pi_1(V)$ is discrete and closed.
By Proposition \ref{prop:prelim_dim}(1), we have only to demonstrate that $\dim V \leq 0$.
We also consider the definable sets
\begin{align*}
& U'=\{(y,x) \in \Gamma\;|\; \text{ there exists an open box }B \subseteq M^2 \text{ such that } (y,x) \in B \text{ and }\\
&\qquad B \cap \Gamma \text{ is the graph of a (not necessarily monotone) definable}\\
&\qquad\text{continuous function defined on }\pi_1(B)\} \text{,}\\
&V' = \Gamma \setminus U' \text{ and }\\
& W= V \setminus V'.
\end{align*}
It is obvious that $V=V' \cup W$.
The inequality $\dim \pi_1(V) \leq 0$ is equivalent to the condition that $\dim \pi_1(V') \leq 0$ and $\dim \pi_1(W) \leq 0$ by Proposition \ref{prop:prelim_dim}(2).
We have already demonstrated that the projection image $\pi_1(V')$ has an empty interior in the course of the proof of \cite[Lemma 4.3]{Fuji4}.
The image is discrete and closed by Lemma \ref{lem:local}.
It means that $\dim \pi_1(V')\leq 0$.

The remaining task is to demonstrate the inequality $\dim \pi_1(W) \leq 0$.
Assume that $\dim \pi_1(W)=1$ for contradiction.
A nonempty interval $J$ is contained in $\pi_1(W)$.
We can find a definable map $f: J \rightarrow W$ such that the composition $\pi_1 \circ f$ is the identity map on $J$ by Lemma \ref{lem:property_d}.
By Proposition \ref{prop:monotonicity}, we may assume that $g:=\pi_2 \circ f:J \rightarrow M$ is monotone and continuous by shrinking the interval $J$ if necessary.
We prove that either one of the following conditions is satisfied by shrinking $J$ if necessary.
\begin{itemize}
\item For any $s \in J$, there are no $t \in M$ with $(s,t) \in \Gamma$ and $t<g(s)$;
\item There exists a definable continuous map $g_1:J \rightarrow M$ such that, for any $s \in J$, $(s,g_1(s)) \in \Gamma$, $g_1(s)<g(s)$ and there are no $t \in M$ with $(s,t) \in \Gamma$ and $g_1(s)<t<g(s)$.
\end{itemize}
Set $J'=\{s \in J\;|\; \nexists t,\ ((s,t) \in \Gamma) \wedge (t<g(s))\}$.
If $J'$ contains an open interval, the first condition is satisfied by replacing $J$ with the open interval contained in $J'$.
Otherwise, $J'$ is discrete and closed by Lemma \ref{lem:local}.
We may assume that $J'=\emptyset$ by shrinking $J$ if necessary.
Consider the definable function $g_1:J \rightarrow M$ given by $g_1(s)=\sup\{t \in M\;|\; (s,t) \in \Gamma \text{ and } t<g(s)\}$.
By Proposition \ref{prop:monotonicity}, we may assume that $g_1$ is continuous by shrinking $J$ once again.
The function $g_1$ obviously satisfies the second condition.

In the same manner, we can demonstrate that either one of  the following conditions is satisfied by shrinking $J$ if necessary.
We omit the proof.
\begin{itemize}
\item For any $s \in J$, there are no $t \in M$ with $(s,t) \in \Gamma$ and $t>g(s)$;
\item There exists a definable continuous map $g_2:J \rightarrow M$ such that, for any $s \in J$, $(s,g_2(s)) \in \Gamma$, $g_2(s)>g(s)$ and there are no $t \in M$ with $(s,t) \in \Gamma$ and $g(s)<t<g_2(s)$.
\end{itemize}
Under this circumstance, the graph of $g$ is contained in both $U$ and $W$, which is a contradiction to the fact that $U \cap W=\emptyset$. 
We have demonstrated that $\dim \pi_1(V) \leq 0$.

When $b \in \pi_1(V)$, we have $b \in \mydcl(A) \subseteq \mydcl(A \cup \{a\})$.
We consider the other case, that is, the case in which $b \not\in \pi_1(V)$.
By modifying the formula $\varphi$ appropriately, we may assume that $\Gamma=U$ without loss of generality.
Consider the set 
\begin{align*}
& X=\{(y,x) \in \Gamma\;|\; \text{ there exists an open box }B \subseteq M^2 \text{ such that } (y,x) \in B \text{ and }\\
&\qquad B \cap \Gamma \text{ is the graph of a constant function defined on }\pi_1(B)\}.
\end{align*}
It is defined by an $\mathcal L(A)$-formula.
If $\pi_2(X)$ is not discrete, we have $\dim \pi_2(X)=1$.
Therefore, we have $\dim \Gamma =2$ by Proposition \ref{prop:prelim_dim}(5).
On the other hand, we have $\dim \Gamma \leq 1$ by the same proposition because $\dim \pi_1(\Gamma) \leq 1$ and $\dim (\Gamma \cap \pi_1^{-1}(y))=0$ for all $y \in \pi_1(\Gamma)$.
A contradiction.
It implies that $\pi_2(X)$ is discrete and closed.

If the point $a$ is contained in $\pi_2(X)$, we have $a \in \mydcl(A)$, which is a contradiction to the assumption. 
We obtain $a \not\in \pi_2(X)$.
Replacing $\Gamma$ with $\Gamma \setminus X$, we may assume that, for any $(y,x) \in \Gamma$, there exists an open box $B$ containing the point $(y,x)$ such that $B \cap \Gamma$ is the graph of a strictly monotone continuous definable function defined on $\pi_1(B)$.
The set $\pi_1(\Gamma \cap \pi_2^{-1}(a))$ is defined by an $\mathcal L(A \cup \{a\})$-formula. 
By the above assumption, it does not contain an open interval.
It is discrete and closed by Lemma \ref{lem:local}.
On the other hand, we have $(b,a) \in \Gamma$.
It implies that $b \in \pi_1(\Gamma \cap \pi_2^{-1}(a))$.
We finally get $b \in \mydcl(A \cup \{a\})$.
\end{proof}

\begin{lemma}\label{lem:infetesimal}
Let $\mathcal L$ be a language containing a binary predicate $<$ and $T$ be a definably complete locally o-minimal $\mathcal L$-theory.
Let $\mathcal M=(M,<,\ldots)$ be a model of $T$.
Let $(a_1,\ldots, a_n)$ be a sequence of $M$.
There exist a sequence of elementary extensions $$\mathcal M_0=\mathcal M \preceq \mathcal M_1 \preceq \cdots \preceq \mathcal M_n$$ and elements $b_i \in M_i$ for all $1 \leq i \leq n$ such that $b_i>a_i$ and the inequalities $t > b_i$ hold true for all $t \in M_{i-1}$ with $t>a_i$.
Here, $M_i$ denotes the universe of $\mathcal M_i$ for all $1 \leq i \leq n$.
\end{lemma}
\begin{proof}
We can easily reduce to the case in which $n=1$ by induction on $n$.
Therefore, we only consider the case in which $n=1$.
Consider the set $a_1^+$ of $\mathcal L(M_0)$-formulas $\Phi(x)$ with a single free variable $x$ satisfying $$\mathcal M \models \exists s\ s>a_1 \wedge (\forall t\  a_1<t<s \rightarrow \Phi(t)).$$
By local o-minimality, it is a complete $1$-type.
There exist an elementary extension $\mathcal M_1=(M_1,\ldots)$ and an element $b_1 \in M_1$ realizing the $1$-type $a_1^+$.
Since $\mlq\mlq\ x>a_1\ \mrq\mrq \in a_1^+$, we have $a_1<b_1$.
For any $t \in M_0$ with $t>a_1$, we also have $\mlq\mlq\ x<t\ \mrq\mrq  \in a_1^+$.
It implies that $b_1<t$.
\end{proof}

\begin{lemma}\label{lem:infetesimal2}
Let $\mathcal L$ be a language containing a binary predicate $<$ and $T$ be a definably complete locally o-minimal $\mathcal L$-theory.
Let $\mathcal M=(M,<,\ldots)$ be a model of $T$ and $\mathcal N=(N,<,\ldots)$ be its elementary extension.
Let $a \in M$ and $b \in N$ be points such that $a<b$ and, $t>b$ for all $t \in M$ with $t>a$.
Then we have $b \not\in \mydcl_{\mathcal N}(M)$.
\end{lemma}
\begin{proof}
Assume that $b \in \mydcl_{\mathcal N}(M)$ for contradiction.
There exists a definable closed discrete subset $X$ of $M$ such that $b \in X^{\mathcal N}$.
By local o-minimality, there exists an element $c \in M$ such that $c>a$ and the open interval $(a,c)$ does not intersect with $X$.
In particular, for any $x \in (a,c)^{\mathcal N}$, we have $x \not\in X^{\mathcal N}$.
It contradicts the condition that $b \in X^{\mathcal N}$ because the inequalities $a<b<c$ hold true by the assumption.
\end{proof}

\begin{theorem}\label{thm:rank}
Let $\mathcal L$ be a language containing a binary predicate $<$ and $T$ be a definably complete locally o-minimal $\mathcal L$-theory.
Let $\mathcal M=(M,\ldots)$ be a model of $T$.
Let $A$ be a subset of $M$ and $X$ be a subset of $M^n$ definable over $A$.
We have $\dim X=\myrk_T^{\mydcl}(X/A)$.
\end{theorem}
\begin{proof}
Let $\mathbb M$ be a monster model of $T$.
We use the same notation for its universe.
We first reduce the theorem to a simple case.
There exists a family $\{C_i\}_{i=1}^N$ of mutually disjoint quasi-special submanifolds which are definable over $A$ such that $X= \bigcup_{i=1}^N C_i$ by Proposition \ref{prop:quasi}.
We obviously have $\myrk(X/A)=\max\{\myrk(C_i/A)\;|\; 1 \leq i \leq N\}$ by the definition of $\myrk$.
On the other hand, we also have $\dim X = \max \{\dim C_i\;|\; 1 \leq i \leq N\}$ by Proposition \ref{prop:prelim_dim}(2).
Hence, we may assume that $X$ is $\pi$-quasi-special manifold without loss of generality, where $d=\dim X$ and $\pi:M^n \rightarrow M^d$ is a coordinate projection.
Permuting the coordinates if necessary, we may assume that $\pi$ is the projection onto the first $d$ coordinates.

We show that $\myrk(X/A) \leq \dim X$.
We have nothing to prove when $d=\dim X=n$.
We consider the other case.
Take an arbitrary point $(a_1,\ldots, a_n) \in X^{\mathbb M}$.
We have only to demonstrate that $a_k \in  \mydcl(\{a_1, \ldots, a_d\} \cup A)$ for any $d<k \leq n$.
Fix $d<k \leq n$.
For any $t \in \mathbb M^d$, the fiber $X_t^{\mathbb M} =\{ y \in \mathbb M^{n-d}\;|\; (t,y) \in X^{\mathbb M}\}$ is discrete and closed because $X$ is a $\pi$-quasi-special submanifold.
Let $\pi_1:M^n \rightarrow M^{d+1}$ be the coordinate projection given by $\pi_1(x_1,\ldots,x_n)=(x_1,\ldots, x_d,x_k)$.
The projection image of a discrete closed definable set is again discrete and closed by Proposition \ref{prop:prelim_dim}(1),(3).
Therefore, the fiber $\pi_1(X^{\mathbb M})_{(a_1,\ldots, a_d)}=\{y \in \mathbb M\;|\; (a_1,\ldots, a_d,y) \in \pi_1(X^{\mathbb M})\}$ is discrete and closed.
The point $a_k$ is obviously contained in $\pi_1(X^{\mathbb M})_{(a_1,\ldots, a_d)}$.
They imply that $a_k \in  \mydcl(\{a_1, \ldots, a_d\} \cup A)$.
We have obtained the inequality $\myrk(X/A) \leq d$.

We next prove the opposite inequality.
Take $a=(a_1,\ldots, a_n) \in X$.
The intersection $V \cap X$ is the graph of a definable continuous map $f=(f_1,\ldots,f_{n-d}):\pi(V) \rightarrow M^{n-d}$ for some open box $V$ containing the point $a$ because $X$ is a $\pi$-quasi-special manifold.
Apply Lemma \ref{lem:infetesimal} to the sequence $(a_1,\ldots, a_d)$.
We get elementary extensions $\mathcal M_0=\mathcal M \preceq \mathcal M_1 \preceq \cdots \preceq \mathcal M_d$ and elements $b_1, \ldots, b_d$ satisfying the conditions in Lemma \ref{lem:infetesimal}.
We may assume that $\mathcal M_i$ are elementary substructures of the monster model $\mathbb M$ for all $1 \leq i \leq d$.
We also have $b_d \not\in \mydcl(M_{d-1})$ by Lemma \ref{lem:infetesimal2}.
It implies that the set $\{b_1,\ldots, b_d\}$ is $\mydcl$-independent over $M$.
It immediately follows from the definition of $\mydcl$-independent that $\{b_1,\ldots, b_d\}$ is $\mydcl$-independent over $A$.
We have $(b_1,\ldots, b_d) \in \pi(X)$ because $\pi(X)$ is open.
Set $b_{j+d}=f_j(b_1,\ldots, b_d)$ for all $1 \leq j \leq n-d$.
Since $\{b_1,\ldots, b_d\}$ is $\mydcl$-independent over $A$, we have $\myrk(\{b_1,\ldots, b_n\}/A) \geq d$.
We also have $(b_1,\ldots, b_n) \in X^{\mathbb M}$ because the graph of $f$ is contained in $X$.
We have demonstrated the inequality $\myrk(X/A) \geq d$.
\end{proof}

We introduce the notion of generic points for future use.

\begin{definition}[Generic point]
Let $\mathcal L$, $T$, $\mathcal M$, $A$ and $X$ be the same as in Theorem \ref{thm:rank}.
Let $\mathbb M$ be a monster model of $T$.
There exists $\overline{c} \in X^{\mathbb M}$ with $\myrk^{\mydcl}(\overline c/A)=\dim X$ by Theorem \ref{thm:rank}.
Such a point $\overline{c} \in X^{\mathbb M}$ is called a \textit{generic point} of $X$.
\end{definition}

\section{First-order topological structure}\label{sec:first-order}

Pillay defined a first-order topological structure and the dimension rank for definable sets.
\begin{definition}
Let $\mathcal L$ be a language and $\mathcal M=(M,\ldots)$ be an $\mathcal L$-structure.
The structure $\mathcal M$ is called a \textit{first-order topological structure} if there exists an $\mathcal L$-formula $\phi(x,\overline{y})$ such that the family $\{\phi(x,\overline{a})\;|\; \overline{a} \subseteq M\}$ is a basis for a topology on $M$.
When $\mathcal M$ is an expansion of a dense linear order, then $\mathcal M$ is a first-order topological structure.

Recall that a set is \textit{constructible} if it is a finite boolean combination of open sets.
We consider the case in which any definable set is constructible such as the case in which the structure is a definably complete locally o-minimal structure \cite[Corollary 3.10]{Fuji4}.
For a definable set $X$, an ordinary valued dimension rank $\mydim(X)$ is defined as follows:
\begin{enumerate}
\item[(1)] If $X$ is nonempty, then $\mydim(X) \geq 0$. Otherwise, set $\mydim(X)=-\infty$.
\item[(2)] If $\mydim(X) \geq \alpha$ for all $\alpha<\delta$, where $\delta$ is limit, then $\mydim(X) \geq \delta$.
\item[(3)] $\mydim(X) \geq \alpha+1$ if there exists a definable closed subset $Y$ of $X$ such that $Y$ has an empty interior in $X$ and $\mydim(Y) \geq \alpha$.
\end{enumerate}
We put $\mydim(X)=\alpha$ if $\mydim(X) \geq \alpha$ and $\mydim(X) \not\geq \alpha+1$.
We set $\mydim(X)=\infty$ when $\mydim(X) \geq \alpha$ for all $\alpha$.
\end{definition}

We begin to prove that $\dim X=D(X)$ when $X$ is a definable set.
\begin{lemma}\label{lem:for_top_st_dim}
Let $\mathcal M=(M,<,\ldots)$ be a definably complete locally o-minimal structure, and let $X$ and $Y$ be definable subsets of $M^n$ with $Y \subseteq X$ and $\dim X=\dim Y$.
Then, the set $Y$ has a nonempty interior in $X$.
\end{lemma}
\begin{proof}
Set $d=\dim X=\dim Y$.
Consider the set $Z=\overline{X \setminus Y} \cap Y$.
We have $\dim Z<d$ by  Proposition \ref{prop:prelim_dim}(4) because $Z \subseteq \partial (X \setminus Y)$.
The set $Z$ is strictly contained in $Y$ by Proposition \ref{prop:prelim_dim}(2).
Take a point $x \in Y \setminus Z$.
By the definition of $Z$, there exists an open box $B$ containing the point $x$ such that $B$ has an empty intersection with $X \setminus Y$.
In other words, the intersection $X \cap V$ is contained  in $Y$.
It means that $Y$ has a nonempty interior in $X$.
\end{proof}

\begin{proposition}\label{prop:top_st_dim}
Let $\mathcal M=(M,<,\ldots)$ be a definably complete locally o-minimal structure and $X$ be a definable set.
We have $\dim X=\mydim(X)$.
\end{proposition}
\begin{proof}
It is obvious when $X$ is an empty set.
We concentrate on the case in which $X$ is not empty.
Let $M^n$ be the ambient space of $X$.
Set $d=\dim(X)$.
We prove the proposition by induction on $d$.
When $d=0$, the definable set $X$ is discrete and closed.
In particular, any nonempty definable subset of $X$ has a nonempty interior in $X$.
It implies that $\mydim(X)=0$.

We next consider the case in which $d>0$.
Let $Y$ be an arbitrary definable closed subset of $X$ having an empty interior in $X$.
We have $\dim Y<d$ by Lemma \ref{lem:for_top_st_dim}.
We get $\mydim(Y)<d$ by the induction hypothesis.
It means that $\mydim(X) \leq d$.

We demonstrate the opposite inequality $\mydim(X) \geq d$.
We have only to construct a definable closed subset $Y$ of $X$ such that $\mydim(Y)=d-1$ and $Y$ has an empty interior in $X$.
We decompose $X$ into quasi-special submanifolds $X_1, \ldots, X_t$ by Proposition \ref{prop:quasi}.
At least one of $X_1,\ldots, X_t$ is of dimension $d$  by Proposition \ref{prop:prelim_dim}(2).
We may assume that $\dim X_1=d$ without loss of generality.
Let $\pi:M^n \rightarrow M^d$ be a coordinate projection such that $X_1$ is a $\pi$-quasi-special submanifold.
Note that $\pi(X_1)$ is an open set, and, as a consequence, it is of dimension $d$.
We may assume that $\pi$ is the projection onto the first $d$ coordinates by permuting the coordinates if necessary.

Set $Z_1=\overline{X \setminus X_1} \cap X_1$.
It is of dimension smaller than $d$ by Proposition \ref{prop:prelim_dim}(2),(4) because $Z_1$ is contained in the frontier $\partial (X \setminus X_1)$.
Its projection image $\pi(Z_1)$ is also of dimension smaller than $d$ by Proposition \ref{prop:prelim_dim}(3).
We have $\dim \pi(X_1) \setminus \pi(Z_1)=d$ by Proposition \ref{prop:prelim_dim}(2).
It implies that $\pi(X_1) \setminus \pi(Z_1)$ is not an empty set.
Take a point $x \in X_1$ with $\pi(x) \not\in \pi(Z_1)$.
By the definition of $Z_1$ and the assumption that $X_1$ is $\pi$-quasi-special submanifold, we can take an open box $U$ in $M^n$ containing the point $x$ such that $U \cap (X \setminus X_1)=\emptyset$ and $U \cap X_1$ is the graph of a definable continuous map defined on $\pi(U)$.
In particular, we have $U \cap X \subseteq X_1$.

Take a nonempty closed box $B$ contained in $\pi(U)$.
The definable set $W=\pi^{-1}(B) \cap U \cap X$ is a definable subset of $X$ which is simultaneously the graph of a definable continuous function defined on $B$.
In particular, the definable set $W$ is closed in $X$.
Take a point $a=(a_1,\ldots,a_d) \in \myint(B)$ and consider the hyperplane $H=\{x=(x_1,\ldots, x_d) \in M^d\;|\; x_d=a_d\}$.
We obviously have $\dim H \cap B = d-1$.
Put $Y= \pi^{-1}(H \cap B) \cap W$.
It is closed in $X$.
Using Proposition \ref{prop:prelim_dim}(5), we get $\dim Y=d-1$.
The induction hypothesis implies that $\mydim(Y)=d-1$.
By the construction, it is obvious that $Y$ has an empty interior in $X$.
We have finally demonstrated the inequality $\mydim(X) \geq d$.
\end{proof}


\end{document}